\documentclass{article}%SCAEOL for online version; SCAE for publication version; SCAES for the paper dedicated to somebody.
\usepackage{graphicx}
\usepackage{amsmath,amsfonts,amsthm,tipa,enumerate,amssymb,multirow,threeparttable,fullpage,hyperref,bookmark}
\usepackage{extarrows}
\usepackage{verbatim}
\usepackage{multirow}
\newcommand\norm[1]{\left\lVert#1\right\rVert}
\newcommand{\0}{{\mathbf{0}}}
    \renewcommand{\a}{{\mathbf{a}}}
\newcommand{\packB}{{\mathcal{B}}}
\newcommand{\packC}{{\mathbb{C}}}

\newcommand{\packF}{{\mathbb{F}}}
\newcommand{\FF}{{\mathcal{F}}}

\newcommand{\ip}{{\mathfrak{p}}}
\newcommand{\iP}{{\mathfrak{P}}}

\newcommand{\LP}{{\mathcal{P}}}
    \renewcommand{\mod}{\ \mathrm{mod}\ }

\newcommand{\N}{{\mathbb{N}}}

    \renewcommand{\O}{{\mathcal{O}}}
\newcommand{\p}{{\mathbf{p}}}
\newcommand{\packQ}{{\mathbb{Q}}}
\newcommand{\R}{{\mathbb{R}}}

\newcommand{\vol}{{\text{vol}}}
    
    \newcommand{\bv}{{\mathbf{v}}}

\newcommand{\Z}{{\mathbb{Z}}}
\newcommand{\Magma}{{\sf Magma}}

\newcommand{\Tr}{\textnormal{Tr}}
\newcommand{\Nor}{\textnormal{N}}

\if@nosecthm
    \newtheorem{theorem}{Theorem}
\else
    \newtheorem{theorem}{Theorem}[section]
\fi

    \newtheorem{corollary}[theorem]{Corollary}

    \newtheorem{example}[theorem]{Example}
    \newtheorem{lemma}[theorem]{Lemma}

    \newtheorem{proposition}[theorem]{Proposition}
    
    \newtheorem{remark}[theorem]{Remark}

\begin{document}

%Basic Information
%\Year{} %
%\Month{}
%\Vol{} %
%\No{} %
%\BeginPage{1} %
%\EndPage{} %
%\AuthorMark{}
%\ReceivedDay{}
%\AcceptedDay{}
%\PublishedOnlineDay{; published online January 22, 2013}
%\DOI{10.1007/s11425-000-0000-0} % The author doesn't need fill in it.

% \title[short text for running head]{full title}{comments for title}
{\begin{center}\Large\bf Dense Packings from Algebraic Number Fields and Codes\footnote{This research began when the author was a PhD candidate at Nanyang Technological University.}\end{center}}{}

% \author[]{Full name}{footnote}
% Remark:  One \author for one author

%\author[1,2]{FIRST1 Last Name1}{}
{\begin{center} {\bf Shantian Cheng}\\Risk Management Institute,\\ National University of Singapore,\\ 21 Heng Mui Keng Terrace,\\
 119613 Singapore\,
%\footnote{State completely without
%test}
\\
{rmicst@nus.edu.sg}
%\author[4]{FIRST3 Last Name3}{}
\end{center}
}

%
%\address{Division of Mathematical Sciences, School of Physical and Mathematical Sciences,\\ Nanyang Technological University, SPMS-MAS-04-20, 21 Nanyang Link, {\rm 637371}, Singapore}
%\address[{\rm2}]{Department of Mathematics, University2, City2 {\rm 100002}, Country2;}
%\address[{\rm3}]{Department of Mathematics, University3, City3 {\rm100003}, Country3;}
%\address[{\rm4}]{College of Science, University4, City4 {\rm100004}, Country4}
%\Emails{ scheng002@e.ntu.edu.sg}\maketitle

%     Abstract is required.

 {\begin{center}
\parbox{14.5cm}{\begin{abstract}
 We introduce a new method from number fields and codes to construct dense packings in the Euclidean spaces. Via the canonical $\mathbb{Q}$-embedding of arbitrary number field $K$ into $\mathbb{R}^{[K:\mathbb{Q}]}$, both the prime ideal $\mathfrak{p}$ and its residue field $\kappa$ can be embedded as discrete subsets in $\mathbb{R}^{[K:\mathbb{Q}]}$. Thus we can concatenate the embedding image of the Cartesian product of $n$ copies of $\mathfrak{p}$ together with the image of a length $n$ code over $\kappa$. This concatenation leads to a packing in the Euclidean space $\mathbb{R}^{n[K:\mathbb{Q}]}$. Moreover, we extend the single concatenation to multiple concatenation to obtain dense packings and asymptotically good packing families. For instance, with the help of {\sf Magma}{}, we construct a $256$-dimensional packing denser than the Barnes-Wall lattice BW$_{256}$.


%  \subjclass is required.

\end{abstract}}\end{center}}

%  Keyword is required.

%%%%%%%%%%%%%%%%%%%%%%%%%%%%%%%%%%%%%%%%%%%%%%%%%%%%%%%%%%%%
\renewcommand{\baselinestretch}{1.2}

\iffalse
\begin{center} \renewcommand{\arraystretch}{1.5}
{\begin{tabular}{lp{0.8\textwidth}} \hline \scriptsize
{\bf Citation:}\!\!\!\!&\scriptsize First1 L N, First2 L N, First3 L N.  SCIENCE CHINA Mathematics  journal sample. Sci China Math, 2013, 56, doi: 10.1007/s11425-000-0000-0\vspace{1mm}
\\
\hline
\end{tabular}}\end{center}
\fi
%%%%%%%%%%%%%%%%%%%%%%%%%%%%%%%%%%%%%%%%%%%%%%%%%%%%%%%%%%%%
%% Text of article.
%%%%%%%%%%%%%%%%%%%%%%%%%%%%%%%%%%%%%%%%%%%%%%%%%%%%%%%%%%%%
%    Section headings
%\baselineskip 11pt\parindent=10.8pt  \wuhao
\section{Introduction}

%Sphere packing is a classical problem concerning how to densely pack non-overlapping equal spheres in $\R^N$. Many subjects, such as discrete geometry, combinatorics, number theory and coding theory, etc. have been involved in this problem.

%Let $\LP$ be the set of centers of packed spheres. The sphere packing problem is to construct packings obtaining large density $\Delta(\LP)$. Lattice packing is a special case in which $\LP$ is a lattice in $\R^N$. The next section begins with a brief review of the basic knowledge and notations. For a detailed survey on the development in this territory, the reader may refer to the book by Conway and Sloane \cite{sloane1999sphere}.

%One classical construction idea is concatenating proper codes with special packings in $\Z^n$. This method may offer new packings denser than the original ones. There are five well-known constructions based on this idea, which are referred as Construction A,B,C(due to Leech and Sloane); D(due to Bos, Conway and Sloane); E(due to Barnes and Sloane). More details about these constructions can be found in \cite{sloane1999sphere,rush2004spherepacking}.

The classical problem of packing non-overlapping equal spheres densely in an $n$-dimensional Euclidean space has attracted the interest of numerous mathematicians for centuries.
%Some explicit fascinating dense constructions and asymptotically good packing families have been provided in the proceedings of the research field.
Many methods and results from different disciplines, such as discrete geometry,
combinatorics, number theory and coding theory, etc. have been involved in
this problem, while some explicit fascinating dense constructions and asymptotically
good packing families have been found. For a detailed survey on the development in
this field, the reader may refer to the books of Cassels~\cite{cassels1997introduction}, Conway and Sloane~\cite{sloane1999sphere}, Zong~\cite{zong1999sphere}.
If the centers of the packed spheres form a discrete additive subgroup of $\R^n$ (lattice), we call it a lattice packing.

Sphere packings are continuous analogues of error-correcting codes in the Hamming
space \cite{elkies2000lattices}. The basic problem of error-correcting codes is to seek the maximum
size of a code given length, alphabet size and minimal Hamming distance, in other
words, dense packing of points in Hamming space such that each pair of distinct
points is separated at least by the minimum Hamming distance. In the digital
world, error-correcting codes are widely employed in information storage and transmission,
for example, the blue ray storage format and the communication between space stations and the Earth. Based on the similarities between sphere packings
and error-correcting codes, the results in sphere packings can potentially contribute
to the development of error-correcting codes.

%Based on the similarities between error-correcting codes and lattices, some constructions of dense lattice or non-lattice packings are inspired by the constructions in coding theory.
Meanwhile, some constructions of dense lattice or non-lattice packings are inspired
by constructions in coding theory. For example, similar to concatenated codes, Leech and Sloane's ``Construction A'' method concatenated certain binary codes together with $2\cdot \Z^n$ to construct new lattice packings (see details in \cite[Chapter 5]{sloane1999sphere}). Bachoc \cite{Bachoc199792} generalized the method to construct modular lattices using codes over finite involution algebras.
In Construction C \cite[Chapter 5]{zong1999sphere}, the binary expansion of the coordinates in $\Z^n$ is considered. A point is a packing center if and only if its first $\ell$ coordinate arrays are codewords in certain $\ell$ binary codes respectively.

Let $\omega=\frac{-1+\sqrt{-3}}{2}$. Xing \cite{xing2008dense} further investigated the concatenating method. Instead of packings in $\Z^n$, he considered the packings in $\O_K^n$, where $\O_K=\Z\left[\omega\right]$ denotes the ring of integers in the number field $\packQ(\sqrt{-3})$. That is, for a non-zero prime ideal $\iP$ of $\O_K$, $\iP^n$ can be embedded as a lattice $L$ in $\R^{2n}$ via the canonical $\packQ$-embedding of the special number field $\packQ(\sqrt{-3})$ into $\R^2$.
Hence any subset $\LP$ of $\iP^n$ can be regarded as a packing in $\R^{2n}$. Then he replaced the binary expansion in Construction C by the $\iP$-adic expansion, and concatenated the lattice $L$ with some special codes over $\O_K/\iP$.
 % and then the lattice can be concatenated with some special codes over $\Z[\omega]/\iP$.
 This method produces several dense  packings in low dimensions attaining the best-known densities
  and also obtains an unconditional bound on the asymptotic density exponent $\lambda\geqslant -1.265$
  (see \cite{xing2008dense}). Cheng \cite{chengFFA} applied this concatenation to multiplicative lattices and improved the asymptotic density of
packing families derived from multiplicative lattices. One natural question is whether
we can use arbitrary number field instead of the special quadratic number field $\packQ(\sqrt{-3})$ to generalize the constructing method.

In this paper, we extend Xing's method to a general level, i.e. we employ the ideals in $\O_K$ instead of $\Z[\omega]$, where $\O_K$ denotes the ring of integers in an arbitrary number field $K$. Suppose the extension degree of $K$ over $\packQ$ is $m$. Minkowski interpreted the elements in $K$ as points in the $m$-dimensional Euclidean space $\R^m$. The interpretation is called ``Minkowski Theory'' in algebraic number theory  (see \cite[Section I.5]{neukirch1999algebraic} or \cite[Section 5.3]{tsfasman1991algebraic}).

By Minkowski's interpretation, we can use the canonical $\packQ$-embedding to construct lattices in $\R^m$ from ideals in $\O_K$. In this way, for a non-zero prime ideal $\ip$ of $\O_K$, the Cartesian product of $n$ copies of $\ip$ can be embedded as a lattice in $\R^{mn}$. Meanwhile, the codes defined over $\O_K/\ip$ with length $n$ can also be embedded as finite subsets of $\R^{mn}$. Then we can proceed with the concatenating method as Construction A on these two kinds of subsets of $\R^{mn}$ to construct dense packings. For instance, we construct a $256$-dimensional packing with center density $\delta$ satisfying $\log_2\delta\geqslant 208.09$, which is larger than $192$ of the Barnes-Wall lattice BW$_{256}$ (see the table of dense packings in \cite[Table 1.3]{sloane1999sphere}). Furthermore, for different choices of
number fields and prime ideals, we also provide several asymptotically good packing
families.

%However, in the imaginary quadratic field $\packQ(\sqrt{-3})$, when a non-zero element is embedded as a vector in $\R^2$, its Euclidean length is just the square root of its field norm in $\packQ(\sqrt{-3})$. But when we consider arbitrary number field $K$, there is no such explicit relations any more. Instead, we give some bounds on the minimum Euclidean distance, and further we apply the software \Magma{} V2.21-4 \cite{MR1484478,MS} to compute the exact minimum Euclidean distance of some instances.

In section \ref{preliminaries}, we recall the formal definitions and necessary properties of
sphere packing densities, algebraic number fields and codes. They play important roles in the main results.
In section \ref{main result}, we describe our
generalized concatenating method in detail. Several examples of dense packings and asymptotically good packing
families for different choices of number fields and prime ideals are provided.
%For instance, we construct one $256$-dimensional packing with center density $\delta$ satisfying $\log_2\delta\geqslant 208.09$, which is larger than $192$ of the Barnes-Wall lattice BW$_{256}$ (see the table of dense packings in \cite[Table 1.3]{sloane1999sphere}). Moreover, for different choices of number fields and prime ideals, we also provide several asymptotically good packing families by using the families of linear codes attaining Gilbert-Varshamov (GV) bound.
The detailed numerical results are provided in Tables \ref{table} - \ref{table4}. In Section \ref{conclusion}, we conclude our main contribution.

%Furthermore, we apply the concatenating method to Rosenbloom and Tsfasman's multiplicative lattices in function fields (see \cite{Multiplicative1990}), and we get two bounds $\lambda\geqslant -1.2653$ (principal lattice case) and $\lambda\geqslant -1.2658$ (congruence lattice case).

\section{Preliminaries}\label{preliminaries}

\subsection{Sphere Packing}

Let $\LP$ be the set of centers of packed spheres and $\packB_N(R)$ be the set
\begin{eqnarray*}
  \left\{\bv=\left(a_1,\cdots,a_N\right)\in \R^N: \norm{\bv}=\sqrt{a_1^2+\cdots+a_N^2}\leqslant R\right\},
\end{eqnarray*}
where $\norm{\bv}$ denotes the Euclidean length of vector $\bv$.
As a sphere packing construction is determined by the arrangement of the sphere centers, we just use $\LP$ to denote the corresponding packing.

\smallskip

For a packing $\LP$, the radius of the equal packed spheres is $r=d_E(\LP)/2$, where $d_E(\LP)$ is the minimum Euclidean distance between two distinct points in $\LP$.
Then the density $\Delta(\LP)$ of packing $\LP$ is defined as
\begin{eqnarray*}
  \Delta(\LP)=\limsup_{R\rightarrow \infty}\dfrac{|\LP\cap \packB_N(R)|\cdot r^N\cdot V_N}{\vol(\packB_N(R+r))},
\end{eqnarray*}
where $V_N$ is the volume of the unit sphere in $\R^N$, that is
\begin{eqnarray*}
  V_N=\begin{cases}
    \dfrac{\pi^{N/2}}{\left(N/2\right)!},&\text{if $N$ is even;}\\
   \dfrac{2^N\pi^{\left(N-1\right)/2}\left(\left(N-1\right)/2\right)!}{N!},& \text{if $N$ is odd.}
  \end{cases}
\end{eqnarray*}

 The sphere packing problem is to construct packings obtaining large density $\Delta(\LP)$. Moreover, the center density $\delta(\LP)$ and density exponent $\lambda(\LP)$ are defined respectively as
\[\delta(\LP)=\dfrac{\Delta(\LP)}{V_N},\qquad \lambda(\LP)=\dfrac 1N \log_2 \Delta(\LP).\]

If $\LP=L$ forms a lattice, the density of lattice packing $L$ can be simplified as
\begin{eqnarray*}
  \Delta(L)=\dfrac{(d_E(L)/2)^NV_N}{\det(L)},
\end{eqnarray*}
where $\det(L)$ is the determinant of $L$. Note that from Minkowski's Convex Body Theorem (see \cite[Theorem 1.4]{MGbook} or \cite[Theorem 4]{DBLP:series/isc/Nguyen10}), for any rank $N$ lattice $L$, the minimum Euclidean distance of $L$ satisfies
\begin{eqnarray}\label{eqn:upper bound of lattice niminum}
  d_E(L)\leqslant \sqrt{N}\det(L)^{1/N}.
\end{eqnarray}

When we explore the asymptotic behavior of a packing family $\FF=\left\{\LP^{(N)}\right\}$ as the dimension $N$ of packing $\LP^{(N)}$ tends to $\infty$, we consider the asymptotic density exponent of the family
\begin{eqnarray}\label{eqn:asymptotic density bound}
  \lambda(\FF)=\limsup_{N\rightarrow \infty}\dfrac 1N \log_2 \Delta\left(\LP^{(N)}\right).
\end{eqnarray}
Note that by Stirling formula, as $N\rightarrow \infty$, we have
\begin{eqnarray}\label{Stirling}
  \log_2 V_N=-\dfrac N2\log_2\dfrac {N}{2\pi e}-\dfrac 12\log_2(N\pi)-\epsilon,
\end{eqnarray}
where $0<\epsilon<(\log_2 e)/(6N)$.

%Minkowski provided a nonconstructive bound that there exists one family $\FF$ such that $\lambda(\FF)\geqslant -1$ (see \cite[p.184]{cassels1997introduction}). However, it is difficult to construct families with $\lambda(\FF)<\infty$ explicitly, while such kind of families are called {\em asymptotically good}.

For the asymptotic aspect, Minkowski gave a nonconstructive bound that asserts that there exists some packing family $\FF$ such that the asymptotic density exponent $\lambda(\FF)$ satisfies $\lambda(\FF)\geqslant -1$ (see \cite[p.184]{cassels1997introduction}). It is of interest to construct families with $\lambda(\FF)>-\infty$ explicitly (such families are called {\em asymptotically good}). Known constructive bounds for families with polynomial or exponential construction complexity in terms of dimension $N$ are listed in the book of Litsyn and Tsfasman \cite[p.628]{tsfasman1991algebraic}.
Rush \cite{RushRD} proved that Minkowski's bound on asymptotic density $\lambda\geqslant -1$ can be attained by lattice packings constructed from codes. More recently, Gaborit and Z\'{e}mor \cite{GZRD} improved the density by a linear factor to the quantity of the form $cn2^{-n}$ for constant $c$.

%The known constructive bounds for families with polynomial or exponential construction complexity in terms of $N$ are listed in book by Litsyn and Tsfasman \cite[p.628]{tsfasman1991algebraic}.
%To our best knowledge, they still remain the best so far.

\subsection{Algebraic Number Fields}

Let $K$ be an algebraic number field of degree $m$ over $\packQ$, and let $\O_K$ be its ring of integers. Suppose $K$ has $s$ real embeddings
\[\rho_1,\cdots,\rho_s:K\hookrightarrow \R,\]
and $t$ pairs of complex conjugate embeddings
\[\sigma_1,\overline{\sigma}_1,\cdots,\sigma_t,\overline{\sigma}_t:K\hookrightarrow \packC.\]
Thus $m=s+2t$.
We consider the canonical embedding $\tau:K\hookrightarrow \R^{s+2t}$, that is, for any $\alpha\in K$, $\tau(\alpha)$ is the following vector
 \begin{eqnarray}
   \quad \left(\rho_1(\alpha),\cdots,\rho_s(\alpha),\sqrt{2}\Re\sigma_1(\alpha),\sqrt{2}\Im\sigma_1(\alpha),\cdots,\sqrt{2}\Re\sigma_t(\alpha),\sqrt{2}\Im\sigma_t(\alpha)\right), \label{tau}
 \end{eqnarray}where $\Re$ denotes the real part and $\Im$ denotes the imaginary part of a complex number. $\tau$ can be directly extended to $K^n\hookrightarrow\R^{mn}$, which is also denoted by $\tau$ in this paper without
causing any confusion.

 Note that in the software \Magma{} V2.21-4 \cite{MR1484478,MS}, the embedding $\tau$ defined above is called the Minkowski map and it is employed to compute the ``minimum norm'' (square of minimum Euclidean distance) of a Minkowski lattice. In Subsection \ref{examples}, we will use the software to compute the density of some explicit examples from our construction.

Set $\Tr=\text{Trace}_{K/\packQ}$ and $\Nor=\text{Norm}_{K/\packQ}$. By the Minkowski Theory \cite[Section I.5]{neukirch1999algebraic}, the embedding $\tau$ maps the non-zero ideals of $\O_K$ to some special lattices in $\R^m$. The following lemma characterizes the determinant of such lattices.
\begin{lemma}[{\cite[Section I.5]{neukirch1999algebraic}}]\label{Lem:L_I}
For any non-zero ideal $I\subseteq \O_K$, $L_I:=\tau(I)$ is a lattice of rank $m$.
The determinant of $L_I$ is
 \[\det L_I=\mathcal{N}(I)\sqrt{|D_K|},\]
 where $D_K$ is the discriminant of $K$ and $\mathcal{N}(I)=[\O_K:I]$ is the absolute norm of $I$.
\end{lemma}

Generally it remains hard to compute the minimum Euclidean distance of lattices. However, we can estimate a lower bound on the Euclidean length of non-zero points in the lattice $L_I$. The proof is sketched here and a similar discussion on general ideal lattices can be found in \cite{bayer2002ideal}. Note that in \cite{bayer2002ideal}, the distance between two points is defined using some special quadratic forms, while in this paper, we only focus on the standard Euclidean distance, which refers to the original meaning of packing in the Euclidean space and brings convenience in computation (\Magma{} uses
standard Euclidean distance as built-in measure).
%Here is an lower bound estimate of the Euclidean distance of non-zero points in $L_I$.

 \begin{lemma}\label{normofideal}
   Let $I\subseteq \O_K$ be a non-zero ideal. For any $0\neq\alpha\in I$, the Euclidean length of the vector $\tau(\alpha)\in L_I$ satisfies
   \begin{eqnarray*}
     \norm{\tau(\alpha)}\geqslant \sqrt{ m}\cdot\left|\Nor(\alpha)\right|^{ 1/m}\geqslant\sqrt{m}\cdot\mathcal{N}(I)^{1/m}.
   \end{eqnarray*} In other words, $d_E(L_I)\geqslant \sqrt{m}\cdot\mathcal{N}(I)^{1/m}$.
%   Specially, if $K$ is totally real, we have %\begin{eqnarray}\label{totallyrealnorm}
%    \norm{\tau(\alpha)}\geqslant \sqrt{m}\cdot %\mathcal{N}(I)^{1/m}.
%  \end{eqnarray}
 \end{lemma}
 \begin{proof}
 The Euclidean length of the vector $\tau(\alpha)$ satisfies
 \begin{eqnarray*}
   \norm{\tau(\alpha)}^2&=&\sum_{i=1}^s\rho_i(\alpha)^2+2\sum_{j=1}^t\sigma_j(\alpha)\bar{\sigma}_j(\alpha)\\
   &\geqslant& (s+2t)\left[\prod_{i=1}^s\rho_i(\alpha)^2\cdot\prod_{j=1}^t\sigma_j(\alpha)\overline{\sigma}_j(\alpha)\cdot\prod_{j=1}^t\sigma_j(\alpha)\overline{\sigma}_j(\alpha)\right]^{\frac{1}{s+2t}}\\
   &=& m\left[\left(\Nor(\alpha)\right)^2\right]^{ 1/m}\geqslant m\cdot \mathcal{N}(I)^{2/m}.
 \end{eqnarray*}
 The last $\geqslant$ becomes an equality if and only if the ideal $I$ is a principal ideal generated by $\alpha$.

 % If $K$ is totally real, the inequalities become
 % \begin{eqnarray*}
%   ||\tau(\alpha)||^2&=&\sum_{i=1}^s\rho_i(\alpha)^2
%   \geqslant s\left[\prod_{i=1}^s\rho_i(\alpha)^2\right]^{\frac{1}{s}}\\
%   &=& m\left[\left(Norm_{K/\packQ}(\alpha)\right)^2\right]^{ 1/m}\\
%   &\geqslant&  mN(I)^{2/m}.
% \end{eqnarray*}
 \end{proof}

 The following lemma is a basic fact in algebraic number theory (see \cite[Section I.2]{neukirch1999algebraic}).
\begin{lemma}
  For any non-zero element $\beta\in \O_K$, we have $|\Nor(\beta)|\geqslant 1$.
\end{lemma}

 From the above lemmas, the minimum Euclidean length of non-zero elements in $L_{\O_K}$ satisfies
 $d_E(L_{\O_K})\geqslant \sqrt{m}$.
 Indeed, as the vector $\tau(1)$ has Euclidean length $\sqrt{m}$, we have
 \begin{eqnarray}\label{eqn:minimum of L0}
   d_E(L_{\O_K})= \sqrt{m}.
 \end{eqnarray}
 %Indeed, if $K$ is totally real number field, $d(L_{\O_K})=\sqrt{m}$; if $K$ is totally imaginary field, $d(L_{\O_K})=\sqrt{\dfrac m2}$ (see \cite[p.568]{tsfasman1991algebraic}).

%%%%%%%%%%%%%%%%%%%%%%%%%%%%%%%%%%%%%%%%%%%%%%%%%%%%%%%%%%%%%%

\subsection{Coding Theory}
We recall some notations and results in coding theory.

For a $q$-ary code $C$, let $n(C),M(C)$ and $d_H(C)$ denote the length, the size, and the minimum Hamming distance of $C$, respectively. Such a code is usually referred to as an $\left(n(C),M(C),d_H(C)\right)$-code.
%For simplicity, we write $C=\left(n(C),M(C),d_H(C)\right)$ to represent $C$ is an $\left(n(C),M(C),d_H(C)\right)$-code.
Moreover, the relative minimum distance $\varrho(C)$ and the rate $R(C)$ are defined respectively as

\begin{eqnarray*}
\varrho(C)=\frac{d_H(C)}{n(C)},\quad  R(C)=\frac{\log_q M(C)}{n(C)}.
\end{eqnarray*}

In particular, if a code $C$ forms a linear space over $\packF_q$, then the code $C$ is called an $\left[n(C),k(C),d_H(C)\right]$-linear code over $\packF_q$, where $k(C):=\log_qM(C)$ is called the dimension of $C$. In this case, the rate $R(C)=\dfrac{k(C)}{n(C)}$.

Let $U_q$ be the set of the ordered pair $(\varrho,R)\in\R^2$, for which there exists a family $\{C_i\}_{i=0}^\infty$ of $q$-ary codes such that $n(C_i)$ increasingly goes to $\infty$ as $i$ tends to $\infty$ and
\begin{eqnarray*}
  \varrho=\lim_{i\rightarrow \infty}\varrho(C_i),\quad \text{and}\quad R=\lim_{i\rightarrow \infty}R(C_i).
\end{eqnarray*}
Here is a result on $U_q$:

\begin{proposition}[{\cite[Section 1.3.1]{tsfasman1991algebraic}} or {\cite[Proposition 3.1]{xing2008dense}}]\label{prop:code asymptotic bound}
  There exists a continuous function $R_q(\varrho)$, $\varrho\in[0,1]$, such that
  \[U_q=\left\{(\varrho,R)\in \R^2:0\leqslant R\leqslant R_q(\varrho),\ 0\leqslant \varrho\leqslant 1 \right\}.\]
  Moreover, $R_q(0)=1$, $R_q(\varrho)=0$ for $\varrho\in\left[(q-1)/q,1\right]$, and $R_q(\varrho)$ decreases on the interval $[0,(q-1)/q]$.
\end{proposition}
For $0<\varrho<1$, the $q$-ary entropy function is given as \begin{eqnarray*}
  H_q(\varrho)=\varrho\log_q(q-1)-\varrho\log_q\varrho-(1-\varrho)\log_q(1-\varrho).
\end{eqnarray*}
The asymptotic Gilbert-Varshamov (GV) bound indicates that
\begin{eqnarray}\label{eqn:GV bound}
  R_q(\varrho)\geqslant R_{GV}(q,\varrho):=1-H_q(\varrho),\quad\text{for all }\varrho\in\left(0,\dfrac {q-1}{q}\right).
\end{eqnarray}
Moreover, for any given rate $R$, there exists a family of linear codes which meets the GV bound (see \cite[Section 17.7]{macwilliams1977theory}).

\section{Main Results}\label{main result}
In this section, we introduce our new construction of dense sphere packings. In particular, several explicit constructions and asymptotically good packing families are provided at the end of this section. Our idea is to apply special concatenating methods on $L_I$ defined in Lemma \ref{Lem:L_I}.

Let $[K:\packQ]=m$ and let $\ip$ be a prime ideal of $\O_K$. Assume that the residue field $F_{\ip}=\O_K/\ip$ of $\ip$ is isomorphic to the finite field $\packF_q$. Let $\tau$ be the canonical embedding defined in Eq.~\eqref{tau}.

For $i\in\N$, let $L_{\ip^i}^n$ denote the Cartesian product of $n$ identical copies of lattice $L_{\ip^i}=\tau(\ip^i)$. Then $L_{\ip^i}^n$ is a lattice of rank $mn$ by Lemma~\ref{Lem:L_I}. Moreover, the determinant satisfies $$\det L_{\ip^i}^n=\left(\det L_{\ip^i}\right)^n,$$ which is straightforward from the definition of lattice determinant. In addition, by the definition of $L_{\ip^i}^n$, the minimum Euclidean distance satisfies
 $$d_E(L_{\ip^i}^n)=d_E(L_{\ip^i}).$$ For simplicity, in this section we write $L_{i}=L_{\ip^{i}}$ for short.

 As $\O_K$ is a Dedekind domain (see \cite[Section I.3]{neukirch1999algebraic} or \cite[Chapter 9]{atiyah1994introduction}), let $\kappa=\O_K/\ip$, then we have $\dim_\kappa \ip^i/\ip^{i+1}=1 $ for all $i\in \N$. We know $\kappa\cong\packF_q$. Thus for $i\in\N$, we can choose the sets $S_i:=\left\{\alpha_1^{(i)}=0,\alpha_2^{(i)},\cdots,\alpha_q^{(i)}\right\}\subseteq \ip^i$ such that  \begin{eqnarray}\label{eqn:definition of SetS}{\alpha}_1^{(i)}\mod \ip^{i+1},\cdots,{\alpha}_q^{(i)}\mod \ip^{i+1}\end{eqnarray} represent the $q$ distinct elements in $\ip^i/\ip^{i+1}$.

%%%%%%%%%%%%%%%%%%%%%%%%%%%%%%%%%%%5
%%%%%%%%%%%%%%%%%%%%%%%%%%%%%%%%%%%55
%Let $[K:\packQ]=m$, $\ip$ be a prime ideal of $\O_K$. Assume that the residue field $F_{\ip}=\O_K/\ip$ of $\ip$ is isomorphic to the finite field $\packF_q$. Let $\tau$ be the canonical embedding defined in Eq. \eqref{tau}.

%For $i\in\N$, let $L_{\ip^i}^n$ denote the Cartesian product of $n$ identical copies of lattice $L_{\ip^i}=\tau(\ip^i)$. Then $L_{\ip^i}^n$ is a lattice of rank $mn$ by Lemma~\ref{Lem:L_I}. Moreover, the determinants satisfy $$\det L_{\ip^i}^n=\left(\det L_{\ip^i}\right)^n,$$ which is straightforward from the definition of lattice determinant. In addition, by the definition of $L_{\ip^i}^n$, the minimum Euclidean distances satisfy
% $$d_E(L_{\ip^i}^n)=d_E(L_{\ip^i}).$$ For simplicity, in this section we write $L_{i}=L_{\ip^{i}}$ for short.

% As $\O_K$ is a Dedekind domain (see \cite[Section I.3]{neukirch1999algebraic} or \cite[Chapter 9]{atiyah1994introduction}), let $\kappa=\O_K/\ip$, then we have $\dim_\kappa \ip^i/\ip^{i+1}=1 $ for all $i\in \N$. We know $\kappa\cong\packF_q$. Thus for $i\in\N$, we can get sets $S_i:=\left\{\alpha_1^{(i)}=0,\alpha_2^{(i)},\cdots,\alpha_q^{(i)}\right\}\subseteq \ip^i$ such that  $${\alpha}_1^{(i)}\mod \ip^{i+1},\cdots,{\alpha}_q^{(i)}\mod \ip^{i+1}$$ represent the $q$ distinct elements in $\ip^i/\ip^{i+1}$.

\subsection{Concatenation with One Code}\label{subsec:one code}

\iffalse
Let $[K:\packQ]=m$, $\ip$ be a prime ideal of $\O_K$. Let $\tau$ be the canonical embedding defined in \eqref{tau}.

Assume that the residue class field $F_{\ip}=\O_K/\ip$ of $\ip$ is isomorphic to the finite field $\packF_q$. Let $\alpha_1=0,\cdots,\alpha_q$ be $q$ elements of $\O_K$ such that $$\alpha_1\mod \ip,\cdots,\alpha_q\mod \ip$$ represent the $q$ distinct elements in $F_{\ip}$.

For a $q$-ary code $C$ with length $n$, we can take the code alphabet set of $C$ to be $S_0:=\left\{\alpha_1=0,\cdots,\alpha_q\right\}\subseteq \O_K$. Thus $C$ can be regarded as a finite subset of $\O_K^n$. The advantage of using $S_0$ as the alphabet set is that it can help us bound the packing radius of our construction. The details are included in the proof of Proposition~\ref{prop:general result}.

For one codeword $\mathfrak{c}=(c_1,\cdots,c_n)\in C\subseteq \O_K^n$, we define $\tau(\mathfrak{c})$ as
\begin{eqnarray*}
  \tau(\mathfrak{c})=\left(\tau(c_1),\cdots,\tau(c_n)\right)\in \R^{mn},
\end{eqnarray*} and then take $\tau(C)=\left\{\tau(\mathfrak{c}):\mathfrak{c}\in C\right\}$.
Moreover, let $L^n_{\ip}$ denote the Cartesian product of $n$ identical copies of lattice $L_{\ip}=\tau(\ip)$.
Then for any subset $\LP$ of $L_\ip^n$, we consider the concatenation
$$\tau(C)+\LP:=\left\{\a+\p\in \R^{mn}:\a\in\tau(C),\p\in \LP\right\}.$$

\fi

We fix one index $i_0\in\N_{\geqslant 1}$ in the discussion here and generalize the result to a family of indices in the next subsection. Let $\LP$ be a subset of $L_{i_0}^n$, which can be regarded as a packing in $\R^{mn}$.

Given a $q$-ary code $C$ with length $n$, in order to perform our concatenation, we take the code alphabet set of $C$ to be $S_{i_0-1}\subseteq \ip^{i_0-1}$ (see Eq.~\eqref{eqn:definition of SetS}). In this way, $C$ can be regarded as a finite subset of $\O_K^n$. The advantage of using $S_{i_0-1}$ as the alphabet set is that it can help us bound the packing radius of our construction. The details are included in the proof of Proposition~\ref{prop:general result}.

We define $\tau(\mathfrak{c})$ for each codeword $\mathfrak{c}=(c_1,\cdots,c_n)\in C\subseteq \O_K^n$ as
\begin{eqnarray*}
  \tau(\mathfrak{c}):=\left(\tau(c_1),\cdots,\tau(c_n)\right)\in \R^{mn},
\end{eqnarray*} and take $\tau(C):=\left\{\tau(\mathfrak{c}):\mathfrak{c}\in C\right\}$.
We consider the concatenation
$$\tau(C)+\LP:=\left\{\a+\p\in \R^{mn}:\a\in\tau(C),\p\in \LP\right\},$$ which corresponds to a packing in $\R^{mn}$. We will analyze the density of such a packing in this subsection.

The following lemma characterizes the minimum Euclidean distance of the concatenation $\tau(C)+\LP$, which will play an important role in later discussions.

\begin{lemma}\label{lem:key distance lemma}
Let $\ip$ be a non-zero prime ideal in $\O_K$ with absolute norm $\mathcal{N}(\ip)=q$ {\em (}i.e., the residue field $F_{\ip}\cong\packF_q${\em )}. For $i_0\in\N_{\geqslant 1}$, let
 \begin{enumerate}[(i)]
   \item $\LP$ be a subset of $L_{i_0}^n$,
   \item $C$ be a $q$-ary $(n,M,d_C)$-code over the alphabet set $S_{i_0-1}$. In addition, $C$ contains the zero codeword.
 \end{enumerate}
 Then the minimum Euclidean distance of the packing $\tau(C)+\LP$ satisfies
    \begin{eqnarray*}
   d_E\left(\tau(C)+\LP\right)\geqslant \min\left\{d_E(L_{i_0-1})\sqrt{d_C},\ d_E(\LP)\right\}.
    \end{eqnarray*}

   In particular, if $d_E(L_{i_0-1}) \sqrt{d_C}\geqslant d_E\left(\LP\right)$, we have exactly $$d_E\left(\tau(C)+\LP\right)=d_E(\LP).$$
\end{lemma}

\begin{proof}
  Let $\tau(\mathfrak{c})+\p$ be a non-zero element in $\tau(C)+\LP$, where $\mathfrak{c}\in C$ and $\p\in \LP$. If $\mathfrak{c}=\0$, then
\begin{eqnarray*}\label{eqn:c=0}
\norm{\tau(\mathfrak{c})+\p}=\norm{\p}
\geqslant d_E(\LP).
%\geqslant d_E\left(L^n_\ip\right)=d_E\left(L_\ip\right).
\end{eqnarray*}

\smallskip

 If $\mathfrak{c}\neq \0$, without loss of generality, we assume $\mathfrak{c}=(c_1,\cdots,c_e,\0)$, where $e$ is the Hamming weight of $\mathfrak{c}$ and $c_i\neq 0$ for $1\leqslant i\leqslant e$. As $\LP\subseteq L_{i_0}^n$, we can further assume $\p=\left(\tau(\rho_1),\cdots,\tau(\rho_n)\right)$, where $\rho_j\in \ip^{i_0}$ for $1\leqslant j\leqslant n$. Thus
 $$\tau(\mathfrak{c})+\p=\left(\tau(c_1+\rho_1),\cdots,\tau(c_e+\rho_e),\tau(\rho_{e+1}),\cdots,\tau(\rho_n)\right).$$As $c_1,\cdots,c_e$ are non-zero elements in $S_{i_0-1}$, i.e., for $1\leqslant i\leqslant e$, $c_i\in\ip^{i_0-1}\setminus \ip^{i_0}$, we have that $c_i+\rho_i$ are non-zero elements in $\ip^{i_0-1}$ for $1\leqslant i\leqslant e$. Thus
\begin{eqnarray*}\label{lenghofcodeword}
\norm{\tau(\mathfrak{c})+\p}^2\geqslant \sum_{i=1}^e\norm{\tau(c_i+ \rho_i)}^2
\geqslant e\cdot d_E^2(L_{i_0-1})\geqslant d_C\cdot d_E^2(L_{i_0-1}).
\end{eqnarray*}

\smallskip

In summary, the minimum Euclidean distance of the packing $\tau(C)+\LP$ satisfies
\begin{eqnarray}\label{eqn:distance inequality}
  d_E\left(\tau(C)+\LP\right)
  \geqslant\min\left\{d_E(L_{i_0-1}) \sqrt{d_C},\ d_E(\LP)\right\}.
 % \geqslant\min\{ q^{1/m}\sqrt{{m}},\ d_\O \sqrt{d_C}\}.
\end{eqnarray}

In particular, if $d_E(L_{i_0-1}) \sqrt{d_C}\geqslant d_E\left(\LP\right)$, then from Eq.~\eqref{eqn:distance inequality} above, we have the minimum distance $d_E\left(\tau(C)+\LP\right)\geqslant d_E(\LP)$. Meanwhile, as $C$ contains the zero codeword, $\LP$ is a subset of $\tau(C)+\LP$, thus $d_E(\LP)\geqslant d_E\left(\tau(C)+\LP\right)$. Finally we have exactly $$d_E\left(\tau(C)+\LP\right)= d_E(\LP).$$
\end{proof}

The following proposition generalizes Proposition $2.3$ of \cite{xing2008dense}, where only the special case $K=\packQ(\sqrt{-3})$ is discussed.
%Note that the requirement that $C$ contains zero codeword is also necessary for Proposition $2.3$ of \cite{xing2008dense} as the property, that any codeword in $C$ has Hamming weight not less than the minimum Hamming distance of $C$, is required.

%The following proposition characterizes the density of the packing corresponding to $\tau(C)+\LP$.
\begin{proposition}\label{prop:general result}
%Let $\LP\subseteq \O_K^n$ be a packing in $\R^{mn}$.
 Under the same assumption on $\LP$ and $C$ as in Lemma \ref{lem:key distance lemma},
 \begin{enumerate}[(i)]
\item if $d_E(L_{i_0-1}) \sqrt{d_C}\geqslant d_E\left(\LP\right)$, then the density of $\tau(C)+\LP$ as a packing in $\R^{mn}$ is at least $M\cdot\Delta\left(\LP\right)$, where $\Delta\left(\LP\right)$ is the density of the packing $\LP$ as a packing in $\R^{mn}$, and $M$ is the size of the code $C$;
\item if $\LP$ is a lattice and $C$ satisfies that for any $\mathfrak{u},\mathfrak{v}\in C$, the sum $\tau(\mathfrak{u})+\tau(\mathfrak{v})$ is equal to $\tau(\mathfrak{w})+\p$ for some $\mathfrak{w}\in C$ and $\p\in \LP$, then $\tau(C)+\LP$ is also a lattice.
\end{enumerate}
\end{proposition}

\begin{proof}
(i)
   We denote the volume of the unit sphere in $\R^N$ by $V_N$ and the sphere of radius $b$ by $\packB_N(b)$. Let $s=\max\{\norm{\tau(\mathfrak{c})}:\mathfrak{c}\in C\}$.
%\begin{eqnarray*}
%=\left\{\left(a_1,\cdots,a_N\right)\in\R^N: \sqrt{a_1^2+\cdots a_N^2}\leqslant b\right\}.
%\end{eqnarray*}
For any $\mathfrak{c}\in C$ and $\p\in\LP\cap \packB_{mn}(b)$, we have $\tau(\mathfrak{c})+\p\in \left(\tau(\mathfrak{c})+\LP\right)\cap\packB_{mn}(b+s)$. This implies that
\begin{eqnarray*}\label{eqn:geometric relation1}
 \left |\left(\tau(\mathfrak{c})+\LP\right)\cap\packB_{mn}(b+s)\right|\geqslant \left|\LP\cap \packB_{mn}(b)\right|.
\end{eqnarray*}
Moreover, as the elements in the alphabet set $S_{i_0-1}$ of $C$ represent the $q$ distinct elements in $\ip^{i_0-1}/\ip^{i_0}$ and $\LP\subseteq L_{i_0}^n$, we have if $\mathfrak{c}_i\neq \mathfrak{c}_j$ then
 \begin{eqnarray*}\label{eqn:geometric relation2}
   \left(\tau(\mathfrak{c}_i)+\LP\right)\cap  \left(\tau(\mathfrak{c}_j)+\LP\right)=\emptyset.
 \end{eqnarray*}

 We write $d=d_E\left(\tau(C)+\LP\right)$ for short and from Lemma \ref{lem:key distance lemma} we have $d=d_E(\LP)$. Finally we obtain
\begin{eqnarray*}
  &&\Delta\left(\tau(C)+\LP\right)\\
  &=&\limsup_{b\rightarrow \infty}\dfrac{\left|\left(\tau(C)+\LP\right)\cap \packB_{mn}\left(b+s\right)\right|\left(d/2\right)^{mn}V_{mn}}{\vol(\packB_{mn}(b+s+d/2)}\\
  &=&\limsup_{b\rightarrow \infty}\dfrac{\left(\sum_{\mathfrak{c}\in C}\left|\left(\tau(\mathfrak{c})+\LP\right)\cap \packB_{mn}\left(b+s\right)\right|\right)\left(d/2\right)^{mn}V_{mn}}{\vol(\packB_{mn}(b+s+d/2)}\\
  &\geqslant&\limsup_{b\rightarrow \infty}\dfrac{|C|\cdot\left|\LP\cap\packB_{mn}(b)\right|\left(d_E(\LP)/2\right)^{mn}V_{mn}}{\vol\left(\packB_{mn}\left(b+s+d/2\right)\right)}\\
  &=&|C|\cdot\Delta\left(\LP\right)=M\cdot\Delta(\LP).
\end{eqnarray*}

\noindent(ii) By the definition of a lattice, this part is true.
\end{proof}

\subsection{Concatenation with a Family of Codes}\label{subsec: family of codes}

For a family of $q$-ary codes $\{C_i\}_{i=0}^{\ell-1}$, where $\ell\in\N_{\geqslant 1}$, we can take the code alphabet set of $C_i$ to be $S_i$ (see Eq.~\eqref{eqn:definition of SetS}).
  Note that the choice of the alphabet set of the codes is only used in the proof. For the computation, we only care about the length, size and minimum Hamming distance of the codes.

     Similar to Subsection \ref{subsec:one code}, for any subset $\LP$ of $L_{\ell}^n$, and a family of $q$-ary codes $\mathcal{C}=\{C_i\}_{i=0}^{\ell-1}$ with length $n$, we consider the concatenation
  \begin{eqnarray}\label{def:packing concatenate family}
    \qquad \tau(\mathcal{C})+\LP:=\sum_{i=0}^{\ell-1}\tau(C_i)+\LP=\left\{\sum_{i=0}^{\ell-1}\a_i+\p\in \R^{mn}:\a_i\in \tau(C_i),\p\in \LP\right\}.
  \end{eqnarray}
   Note that for each $0\leqslant i\leqslant \ell-1$, as the code alphabet set of $C_i$ is $S_i$, the set $\tau(C_i)$ is a finite subset of $L_{i}^n$. Meanwhile $\LP$ is a subset of $L_{\ell}^n$. As for any $0\leqslant i\leqslant \ell$, the lattice $L_{i}^n$ is a sublattice of $L_{\O_K}^n$, the concatenation Eq.~\eqref{def:packing concatenate family} makes sense within $L_{\O_K}^n$.

%For simplicity, in this subsection we write $d_\O$ short for the minimum Euclidean distance of lattice $L_{\O_K}$.
\begin{proposition}\label{Strong}
Let $\ip$ be a non-zero prime ideal in $\O_K$ with absolute norm $\mathcal{N}(\ip)=q$. For $\ell\in\N_{\geqslant 1}$, let
\begin{enumerate}[(i)]
\item $\LP$ be a subset of $L_\ell^n$,
\item $\mathcal{C}=\left\{C_i\right\}_{i=0}^{\ell-1}$ be a family of $q$-ary codes, where $C_i$ is an $\left(n,M_i,d_{C_i}\right)$-code, the alphabet set of $C_i$ is $S_i$, and $d_{C_i}\geqslant \left\lceil\dfrac{d_E^2(\LP)}{d_E^2(L_i)}\right\rceil$. In addition, for each $0\leqslant i\leqslant \ell-1$, $C_i$ contains the zero codeword.
\end{enumerate}

Then the density of $\tau(\mathcal{C})+\LP$ as a packing in $\R^{mn}$ is at least $\Delta(\LP)\cdot\prod_{i=0}^{\ell-1}M_i$.
\end{proposition}

\begin{proof}

%%%%%%%%%%%%%%%%%%%%%%%%%%%%%%%%%%%%%%%%%%%
%%%%%%%%%%%%%%%%%%%%%%%%%%%%%%%%%%%%%%%%%%%%%5
Let $\LP_\ell:=\LP\subseteq L_\ell^n$. For $i$ from $\ell-1$ to $0$, we can recursively define
$$\LP_i:=\tau(C_{i})+\LP_{i+1}\subseteq L_i^n.$$ Note that $\LP_0=\tau(\mathcal{C})+\LP$.

We claim that for $0\leqslant i\leqslant \ell-1$,
\begin{eqnarray}\label{eqn:induction assumption}
  d_E(\LP_i)=d_E(\LP_{i+1})\quad \text{and}\quad \Delta(\LP_{i})\geqslant M_i\cdot \Delta(\LP_{i+1}).
\end{eqnarray}

We use induction on $k=\ell-i$ to prove the claim. When $k=1$, $i=\ell-1$, as the minimum Hamming distance of $C_{\ell-1}$ satisfies $d_{C_{\ell-1}}d_E^2(L_{\ell-1})\geqslant d_E^2(\LP)=d_E^2(\LP_\ell)$, by Lemma \ref{lem:key distance lemma} and Proposition \ref{prop:general result} (i), we have
\begin{eqnarray*}
 \quad  d_E(\LP_{\ell-1})=d_E\left(\tau(C_{\ell-1})+\LP_\ell\right)=d_E(\LP_\ell)\quad \text{and}\quad \Delta(\LP_{\ell-1})\geqslant M_{\ell-1}\cdot \Delta(\LP_{\ell}).
\end{eqnarray*}

\smallskip

Suppose for all $k$ in the range $1\leqslant k<k_0+1\leqslant \ell$, the statement Eq.~\eqref{eqn:induction assumption} is true, i.e., $i=\ell-k$,
  \begin{eqnarray}\label{eqn:induction step}
  d_E(\LP_{\ell-k})=d_E(\LP_{\ell-k+1})\quad \text{and}\quad \Delta(\LP_{\ell-k})\geqslant M_{\ell-k}\cdot \Delta(\LP_{\ell-k+1})
  \end{eqnarray}
   is true. By induction, we need to prove
  \begin{eqnarray*}
  \quad d_E(\LP_{\ell-k_0-1})=d_E(\LP_{\ell-k_0})\quad \text{and}\quad \Delta(\LP_{\ell-k_0-1})\geqslant M_{\ell-k_0-1}\cdot\Delta(\LP_{\ell-k_0}).
  \end{eqnarray*}
From Eq.~\eqref{eqn:induction step}, we have
\begin{eqnarray*}
  d_E(\LP_{\ell-k_0})=d_E(\LP_{\ell-k_0+1})=\cdots=d_E(\LP_{\ell-1})=d_E(\LP_{\ell})=d_E(\LP).
\end{eqnarray*}
Besides, the minimum Hamming distance of $C_{\ell-k_0-1}$ satisfies
\begin{eqnarray*} d_{C_{\ell-k_0-1}}d_E^2(L_{\ell-k_0-1})\geqslant d_E^2(\LP)=d_E^2(\LP_{\ell-k_0}).
\end{eqnarray*}
Then by Lemma \ref{lem:key distance lemma} and Proposition \ref{prop:general result} (i), as $\LP_{\ell-k_0-1}=\tau(C_{\ell-k_0-1})+\LP_{\ell-k_0}$, we get
\begin{eqnarray*}
  \quad d_E(\LP_{\ell-k_0-1})=d_E(\LP_{\ell-k_0})\quad \text{and}\quad \Delta(\LP_{\ell-k_0-1})\geqslant M_{\ell-k_0-1}\cdot \Delta(\LP_{\ell-k_0}).
\end{eqnarray*}Thus we have proved the claim Eq.~\eqref{eqn:induction assumption}. From the claim, we easily obtain $\Delta(\LP_0)\geqslant \Delta(\LP)\cdot\prod_{i=0}^{\ell-1}M_i$.
\end{proof}

\begin{remark}
  From the above proposition, given a dense packing $\LP$ from the canonical embedding Eq.~\eqref{tau} on some prime ideal in a certain algebraic number field, we can concatenate several codes satisfying the requirements in Proposition \ref{Strong} to $\LP$ to obtain some new denser packings. The density increases by a ratio larger than the product of the sizes of these codes.
\end{remark}

 Note that generally it is hard to determine the minimum Euclidean distance of $\LP\subseteq L_{\ell}^n\subseteq \R^{mn}$. However, we can consider special choices of $\LP$ with some algebraic structure. The following corollary,  which plays a crucial role in Subsection \ref{examples} for Examples \ref{example1} - \ref{lastexample}, considers the case $\LP=L_{\ell}^n$. The advantage is that for all $i\in \N$, we have $d_E(L_i^n)=d_E(L_{i})$. Instead of searching for the minimum Euclidean distance of $L_i^n$ in $\R^{mn}$, we can search for the minimum Euclidean distance of $L_i$ in $\R^{m}$. In our examples, $m$ is small such that \Magma{} can be used to compute the minimum Euclidean distance.

\begin{corollary}\label{coro:LP=L_l}
  Let $\ip$ be a non-zero prime ideal in $\O_K$ with absolute norm $\mathcal{N}(\ip)=q$. For $\ell\in\N_{\geqslant 1}$, let
$\mathcal{C}=\left\{C_i\right\}_{i=0}^{\ell-1}$ be a family of $q$-ary codes, where $C_i$ is an $\left(n,M_i,d_{C_i}\right)$-code, the alphabet set of $C_i$ is $S_i$, and $d_{C_i}\geqslant \left\lceil\dfrac{d_E^2(L_\ell)}{d_E^2(L_i)}\right\rceil$. In addition, for each $0\leqslant i\leqslant \ell-1$, the code $C_i$ contains the zero codeword.

Then the density of $\tau(\mathcal{C})+L_\ell^n$ as a packing in $\R^{mn}$ satisfies $$\Delta(\tau(\mathcal{C})+L_\ell^n)\geqslant \dfrac{\left(d_E(L_\ell)/2\right)^{mn}V_{mn}}{\left(q^\ell\sqrt{|D_K|}\right)^n}\cdot\prod_{i=0}^{\ell-1}M_i,$$where $D_K$ is the discriminant of $K$ and $V_{mn}$ is the volume of the unit sphere in $\R^{mn}$. Moreover, the center density $\delta=\delta(\tau(\mathcal{C})+L_\ell^n)$ satisfies
$$\log_2\delta\geqslant mn\log_2d_E(L_\ell)-mn-n\ell\log_2q-\frac n2\log_2|D_K|+\sum_{i=0}^{\ell-1}\log_2M_i.$$

In particular, if for all $0\leqslant i\leqslant \ell-1$, the code $C_i$ is a linear code of dimension $k_i=\log_qM_i$, then
\begin{eqnarray*}
  \log_2\delta\geqslant mn\log_2d_E(L_\ell)-mn-n\ell\log_2q-\frac n2\log_2|D_K|+\log_2q\cdot\sum_{i=0}^{\ell-1}k_i.
\end{eqnarray*}
\end{corollary}

\begin{proof}
  Note that $\mathcal{N}(\ip^\ell)=q^\ell$ and then by Lemma \ref{Lem:L_I} $$\det L_{\ell}^n=\left(\det L_{\ell}\right)^n=\left(q^\ell\sqrt{|D_K|}\right)^n.$$
\end{proof}

\begin{remark}
  For concatenation with finite codes, we care more about the dimension of the packing as we need to compare the new construction with the known good packings in the Euclidean space with the same dimension. Hence we fix the length $n$ of the codes first. As for $0\leqslant i\leqslant \ell-1$, the lattice $L_i$ is a sublattice of $L_0$, we have $d_E(L_i)\geqslant d_E(L_0)$. Thus the length $n$ satisfies $n\geqslant \left\lceil\dfrac{d_E^2(L_\ell)}{d_E^2(L_0)}\right\rceil\geqslant \left\lceil q^{2\ell/m}\right\rceil$ by Eq.~\eqref{eqn:minimum of L0} and Lemma \ref{normofideal}. Therefore, we can concatenate at most $\ell=\left\lfloor\dfrac m2\log_q n\right\rfloor$ codes to $L_{\ell}^n$. This bound will be used in the computation of Subsection \ref{examples}.
\end{remark}

\subsection{Asymptotic Properties}

In this subsection, we show that our construction will lead to asymptotically good packing families, which means we can obtain several constructive bounds on the asymptotic density exponent Eq.~\eqref{eqn:asymptotic density bound}.

We will employ families of linear codes that meet the GV bound. As there exist only exponential time algorithms or randomized polynomial algorithms to construct such families (see \cite[Section 17.7]{macwilliams1977theory}), our corresponding asymptotic density exponent bounds are exponential constructive bounds.

From Eq.~\eqref{eqn:upper bound of lattice niminum} and Lemma \ref{normofideal}, for $0\leqslant i\leqslant \ell-1$, we know
\begin{eqnarray}\label{eqn:the upper bound of hamming weight}
  \dfrac{d_E^2(L_\ell)}{d_E^2(L_i)} \leqslant q^{2(\ell-i)/m}|D_K|^{1/m}.
\end{eqnarray}

Different from the finite concatenation in Subsection \ref{subsec: family of codes}, for asymptotic results, we first fix $\ell$ and construct the packing for each $\ell\in\N_{\geqslant1}$. The length $n_\ell$ of the codes is set depending on $\ell$.

 We uniformly set $n_\ell=\left\lceil q^{2\ell/m}|D_K|^{1/m}\right\rceil$. Then by Eq.~\eqref{eqn:the upper bound of hamming weight}, $n_\ell\geqslant \left\lceil \dfrac{d_E^2(L_\ell)}{d_E^2(L_i)}\right\rceil$ for all $0\leqslant i \leqslant \ell-1$. Based on the GV bound Eq.~\eqref{eqn:GV bound}, for $0\leqslant i\leqslant\ell-1$, we can choose $q$-ary \begin{eqnarray*}
  \left[n_\ell,k^{(n_\ell)}_i,\left\lceil \dfrac{d_E^2(L_\ell)}{d_E^2(L_i)}\right\rceil\right]
\end{eqnarray*} linear code $C^{(n_\ell)}_{i}$ such that
the rate is $$\dfrac{k^{(n_\ell)}_i}{n_\ell}\geqslant R_{GV}(q,\varrho^{(n_\ell)}_i)=1-H_q(\varrho^{(n_\ell)}_i),$$where the relative minimum distance satisfies $$\varrho^{(n_\ell)}_i=\dfrac{\left\lceil \dfrac{d_E^2(L_\ell)}{d_E^2(L_i)}\right\rceil}{n_\ell}
%\sim  \dfrac{1}{q^{2i/m}},\quad \text{as $\ell\rightarrow\infty$}
.$$

Let $\mathcal{C}^{(n_\ell)}:=\left\{C_i^{(n_\ell)}\right\}_{i=0}^{\ell-1}$. For each $n_\ell$, define a packing \begin{eqnarray*}
  \LP^{(n_\ell)}:=\tau(\mathcal{C}^{(n_\ell)})+L_{\ell}^{n_\ell}
\end{eqnarray*}as in Eq.~\eqref{def:packing concatenate family}.

 Next we take $\ell$ increasingly to $\infty$. The following proposition describes the asymptotic density exponent of the packing family $\mathcal{F}=\left\{\LP^{(n_\ell)}\right\}_{\ell\rightarrow \infty}$, where $n_\ell=\left\lceil q^{2\ell/m}|D_K|^{1/m}\right\rceil$ tends to $\infty$ as $\ell$ goes to $\infty$.

\begin{proposition}\label{prop:asymptotic result}
 The asymptotic density exponent of the family $\mathcal{F}$ satisfies
 \begin{eqnarray}\label{eqn:asymptotic density of concatenation family}
   \lambda(\mathcal{F})&\geqslant& -1-\frac {1}{2m}\log_2 |D_K|-\frac 12\log_2 \frac{m}{2\pi e}\nonumber\\
  &&\quad +\limsup_{\ell\rightarrow \infty}\left(\log_2 d_E(L_{\ell})-\frac 12\log_2 n_\ell -\frac {\log_2 q}{m}\sum_{i=0}^{\ell-1}H'_q\left(\varrho_i^{(n_\ell)}\right)\right),
 \end{eqnarray}
 where $H_q'(\varrho)=H_q(\varrho)$ for $0<\varrho<\dfrac {q-1}{q}$ and $H_q'(\varrho)=1$ for $\dfrac{q-1}{q}\leqslant \varrho\leqslant 1$.
\end{proposition}

\begin{proof}

%We write $\log$ short for $\log_2$.
By Corollary \ref{coro:LP=L_l},
  \begin{eqnarray*}
   &&\lambda(\mathcal{F})\geqslant \limsup_{\ell\rightarrow \infty}\frac {1}{mn_\ell}\log_2\left(\dfrac{\left(d_E(L_{\ell})/2\right)^{mn_\ell}V_{mn_\ell}}{\left(q^{\ell}\sqrt{|D_K|}\right)^{n_\ell}}\cdot\prod_{i=0}^{\ell-1}M_i^{(n_\ell)}\right)\\
   &\ &=-1-\frac {1}{2m}\log_2 |D_K|+ \\
   &&\ +\limsup_{\ell\rightarrow \infty}\left(\log_2 d_E(L_{\ell})-\frac{\ell} {m}\log_2 q+\frac {1}{mn_\ell}\log_2 V_{mn_\ell}+\sum_{i=0}^{\ell-1}\frac{\log_2 q\cdot k_i^{(n_\ell)}}{mn_\ell}\right).
  \end{eqnarray*}
By Eq.~\eqref{Stirling},
\begin{eqnarray*}
  &&\limsup_{\ell\rightarrow \infty}\left(\frac {1}{mn_\ell}\log_2 V_{mn_\ell}+\sum_{i=0}^{\ell-1}\frac{\log_2 q\cdot k_i^{(n_\ell)}}{mn_\ell}\right)\\
  &=&\limsup_{\ell\rightarrow \infty}\left(-\frac 12\log_2\frac{mn_\ell}{2\pi e}-\frac{1}{2mn_\ell}\log_2 mn_\ell\pi+\dfrac{\log_2 q}{m}\sum_{i=0}^{\ell-1}\frac{ k_i^{(n_\ell)}}{n_\ell}\right)\\
  &\geqslant&-\frac 12\log_2 \frac{m}{2\pi e}+\limsup_{\ell\rightarrow \infty}\left(-\frac 12\log_2 n_\ell+\frac {\log_2 q}{m}\sum_{i=0}^{\ell-1}R_{GV}\left(q,\varrho_i^{(n_\ell)}\right)\right)\\
   &=&-\frac 12\log_2 \frac{m}{2\pi e}+\limsup_{\ell\rightarrow \infty}\left( -\frac 12\log_2 n_\ell+\frac {\log_2 q}{m}\sum_{i=0}^{\ell-1}\left(1-H'_q\left(\varrho_i^{(n_\ell)}\right)\right)\right)\\
  &=&-\frac 12\log_2 \frac{m}{2\pi e}+\limsup_{\ell\rightarrow \infty}\left(\frac {\ell} m\log_2 q-\frac 12\log_2 n_\ell-\frac {\log_2 q}{m}\sum_{i=0}^{\ell-1}H'_q\left(\varrho_i^{(n_\ell)}\right)\right).
\end{eqnarray*}
In summary, we get Eq.~\eqref{eqn:asymptotic density of concatenation family}.
\end{proof}

\subsection{Examples}\label{examples}

We use some explicit examples to illustrate our new construction. Here \Magma{} V2.21-4 \cite{MS,MR1484478} will be employed to obtain the numerical results.

For the finite concatenation, we use the linear codes from \cite{Grassl:codetables} which have an explicit construction. For the asymptotic result, we set sufficiently large $\ell$ to approximate the limit.

\begin{example}\label{example1}
   Consider the number field $K=\packQ[\alpha]$, where $\alpha$ is a root of the irreducible polynomial $$f(x)=x^4-x^3-x^2+x+1\in\packQ[x].$$ $K/\packQ$ is a totally complex number field with extension degree $[K:\packQ]=4$ and absolute discriminant $|D_K|=117$. The absolute discriminant is the smallest one of all totally complex number fields with degree $4$ (see \cite{Odlyzko1990}).
\begin{enumerate}[(i)]
  \item We consider a prime ideal $\ip$ lying over $3\in\Z$. The \Magma{} code of this example is listed in Appendix \ref{magma code}, while for other examples, the \Magma{} code can be modified from this template.

      The key outputs of Appendix \ref{magma code} are listed here:
       \begin{verbatim}
   The degree of K=Q[x]/(x^4 - x^3 - x^2 + x + 1) is m=4;
   The absolute value of the discriminant of K is |d|=117;
   P is a Prime Ideal of O
   Two element generators:
   [3, 0, 0, 0]
   [2, 1, 1, 0] lying over 3 with norm q=9;
       \end{verbatim}
      The above statement means $\ip=(3,2+\alpha+\alpha^2)$.
      \begin{verbatim}
   Finite Concatenation:
   We can concatenate at most 3 linear codes to the lattice
    constructed by 64 copies of L_{P^3},
   whose Hamming weights are required respectively at least
   27;9;3;
   \end{verbatim}

   Referring to \cite{Grassl:codetables}, as the norm of $\ip$ is $9$, we can use the existing $9$-ary linear codes $C_0, C_1, C_2$ with parameters
      \begin{eqnarray*}
         \left[64,25,27\right],\
        \left[64,49,9\right],\
        \left[64,61,3\right]
      \end{eqnarray*}respectively.
     The sum of the dimensions is $135$. Considering the concatenation in Corollary \ref{coro:LP=L_l}, we obtain a packing with dimension $4\times 64=256$, whose center density $\delta$ satisfies
   \begin{verbatim}
   Our packing is in dimension 256 with Log_2(center density)
    at least  208.088204168043224246772217517
   \end{verbatim}
   Note that our packing is denser than the Barnes-Wall lattice BW$_{256}$, whose density is listed in the table of dense sphere packings \cite[Table 1.3]{sloane1999sphere}.

   Using this ideal to construct a packing family as in Proposition \ref{prop:asymptotic result}, we get the following result.
\begin{verbatim}
   Asymptotic result:
   The asymptotic density exponent of the packing family is
    at least -1.442426720
      \end{verbatim}
 Note that the above result means our packing family is asymptotically good.
  \item For this field, we similarly analyze other prime ideals and list part of our numerical results in Table~\ref{table}, where $\delta$ is the center density and $\lambda$ is the asymptotic density of the corresponding packing family.

     \begin{table}[htbp]
     \centering
\caption{Examples constructed from $\packQ[x]/\left(x^4-x^3-x^2+x+1\right)$}\label{table}
\begin{threeparttable}

\begin{tabular}{|c|c|c|c|c|}

\hline
 $\ip$& $q$&dimension & $\log_2\delta$ at least &$\lambda$ at least\\
 \hline
 \multirow{3}{*}{$(3,2+\alpha+\alpha^2)$}&\multirow{3}{*}{$9$}&$180$&$108.52$&\multirow{3}{*}{$-1.442$}\\

&&$256$&$208.09\tnote{1}$&\\
&&$512$&$590.52$&\\
\hline
\multirow{2}{*}{$(7,4+\alpha)$}&\multirow{2}{*}{$7$}&$256$&$190.63$&\multirow{2}{*}{$-1.453$}\\
&&$400$&$410.15$&\\
\hline
\end{tabular}
\begin{tablenotes}
        \footnotesize
        \item[1] The packing can be explicitly constructed with $\log_2 \delta$ larger than $192$ given by the Barnes-Wall lattice BW$_{256}$.
      \end{tablenotes}
      \end{threeparttable}

\end{table}

\end{enumerate}

\end{example}

\begin{example}
  Consider the number field $K=\packQ[\alpha]$, where $\alpha$ is a root of the irreducible polynomial $$f(x)=x^3+x^2-2x-1\in\packQ[x].$$ $K/\packQ$ is a totally real number field with extension degree $[K:\packQ]=3$ and absolute discriminant $|D_K|=49$. The absolute discriminant is the smallest one of all totally real cubic number fields (see \cite{Odlyzko1990}). The numerical results on our examples are partially listed in Table \ref{table2}.
  \begin{table}[htbp]
  \centering
\caption{Examples constructed from $\packQ[x]/\left(x^3+x^2-2x-1\right)$}\label{table2}
\begin{threeparttable}
\begin{tabular}{|c|c|c|c|c|}

\hline
 $\ip$& $q$&dimension & $\log_2\delta$ at least &$\lambda$ at least\\
 \hline
{$(2)$}&{$8$}&$255$&$134.46$&{$-1.628$}\\
\hline
\multirow{2}{*}{$(7,5+\alpha)$}&\multirow{2}{*}{$7$}&$192$&$83.68$&\multirow{2}{*}{$-1.585$}\\
&&$255$&$157.63$&\\
\hline
\end{tabular}
      \end{threeparttable}

\end{table}

\end{example}
\begin{example}
  Consider the number field $K=\packQ[\alpha]$, where $\alpha$ is a root of the irreducible polynomial $$f(x)=x^3+x^2-1\in\packQ[x].$$ $K/\packQ$ is a number field with extension degree $[K:\packQ]=3$ and absolute discriminant $|D_K|=23$. The absolute discriminant is the smallest one of all cubic number fields (see \cite{Odlyzko1990}). The numerical results on our examples are partially listed in Table \ref{table3}.
  \begin{table}[htbp]
  \centering
\caption{Examples constructed from $\packQ[x]/\left(x^3+x^2-1\right)$}\label{table3}
\begin{threeparttable}
\begin{tabular}{|c|c|c|c|c|}

\hline
 $\ip$& $q$&dimension & $\log_2\delta$ at least &$\lambda$ at least\\
 \hline
 \multirow{2}{*}{$(2)$}&\multirow{2}{*}{$8$}&$96$&$24.70$&\multirow{2}{*}{$-1.429$}\\

&&$192$&$115.40$&\\
\hline
\multirow{2}{*}{$(5,2+\alpha)$}&\multirow{2}{*}{$5$}&$150$&$69.47$&\multirow{2}{*}{$-1.445$}\\
&&$180$&$101.01$&\\
\hline
{$(7,11+\alpha)$}&{$7$}&$255$&$187.32$&{$-1.430$}
\\
\hline
\end{tabular}

      \end{threeparttable}

\end{table}

\end{example}

\begin{example}\label{lastexample}
  Consider the number field $K=\packQ[\alpha]$, where $\alpha$ is a root of the irreducible polynomial $$f(x)=x^6+x^3+1\in\packQ[x].$$ $K/\packQ$ is a number field with extension degree $[K:\packQ]=6$ and absolute discriminant $|D_K|=19683$. The polynomial $f(x)=x^6+x^3+1$ is the $9$th cyclotomic polynomial over $\packQ$. The numerical results on our examples are partially listed in Table \ref{table4}.
  \begin{table}[htbp]
  \centering
\caption{Examples constructed from $\packQ[x]/\left(x^6+x^3+1\right)$}\label{table4}
\begin{threeparttable}
\begin{tabular}{|c|c|c|c|c|}

\hline
 $\ip$& $q$&dimension & $\log_2\delta$ at least &$\lambda$ at least\\
 \hline
\multirow{2}{*}{$(3,2+\alpha)$}&\multirow{2}{*}{$3$}&$180$&$109.71$&\multirow{2}{*}{$-1.868$}\\
&&$192$&$122.72$&\\
\hline

\end{tabular}
      \end{threeparttable}

\end{table}

\end{example}

\section{Conclusion}\label{conclusion}

This paper provides a new method to construct dense packings using canonical $\packQ$-embedding of algebraic number fields, and special codes over the residue field of some prime ideals.
%With the help of \Magma{} V2.21-4, several explicit constructions were provided in Example \ref{example1} - \ref{lastexample}. Especially, in $\R^{256}$, a packing denser than the Barnes-Wall lattice BW$_{256}$ was obtained. Moreover, this method can be utilized to construct asymptotically good packing families. For each number field and prime ideal in Table \ref{table} - \ref{table4}, one lower bound on the asymptotic density exponent of the corresponding packing family is exhibited.
With the help of \Magma{} V2.21-4, several explicit constructions were provided in Examples \ref{example1} - \ref{lastexample}. Especially, in $\R^{256}$, a packing denser than the Barnes-Wall lattice BW$_{256}$ was obtained. Moreover, this method can be utilized to construct asymptotically good packing families. For each number field and prime ideal in Tables \ref{table} - \ref{table4}, a lower bound on the asymptotic density exponent of the corresponding packing family was provided.

\section*{Acknowledgements}{

Financially, the accomplishment of the first version had been partially supported by Nanyang Technological University under NTU Research Scholarship, when the author was a PhD candidate, and partially supported by Yujie Nan when the author was unemployed till he joined RMI, NUS.

The author sincerely thanks his PhD supervisors, San Ling and Chaoping Xing, for introducing him to this topic, especially for the invaluable suggestions and comments from Chaoping Xing which make the author's initial idea become mature.
%The author would also like to thank the editors and reviewers for their time and helps on this paper.
}
\appendix
\section{\Magma\ Code for Example \ref{example1}}\label{magma code}
{\footnotesize
 \begin{verbatim}
  R:=RealField(10); Q:=RationalField();
  W<x>:=PolynomialRing(Q); f:=x^4-x^3-x^2+x+1;
              /*Input a polynomial over rational number field Q*/
  if IsIrreducible(f) then   /* Test whether f is irreducible*/
   K<a>:=NumberField([f]); O:=MaximalOrder(K); m:=AbsoluteDegree(K);
   printf"The degree of K=Q[x]/(%o) is m=%o;\n",f,m;
   d:=AbsoluteDiscriminant(K);
   printf"The absolute value of the discriminant of K \
    is |d|=%o;\n",d;

   p:=3;     /*Choose p=3 as a base prime number*/
   J:=Decomposition(O,p); P:=J[1,1]; q:=Norm(P);
   printf"P is a %o lying over %o with norm q=%o;\n",P,p,q;

   printf"\nFinite Concatenation:\n";
   n:=64; l:=Floor(m/2*Log(n)/Log(q));
          /*Set the length of the codes*/
   L:=MinkowskiLattice(P^l); b:=Minimum(L);
   printf"We can concatenate at most %o linear codes to \
    the lattice constructed by %o copies of L_{P^%o},\n \
    whose Hamming weights are required respectively at least \
    \n",l,n,l;

   for t in [0..l-1] do
     Z:=MinkowskiLattice(P^t); h:=Ceiling(Minimum(L)/Minimum(Z));
     printf"%o;",h;
   end for;
   printf"\n";

   printf"Refer to the Codetable.de to get the largest dimension \
    of the corresponding linear codes with length %o; \n",n;
   fromtable:=135;
   printf"Here the sum of dimensions T=%o;\n",fromtable;
           /*For this example, the sum of the dimensions is 135*/
   c:=(m*n*Log(b^0.5)-n*l*Log(q)-0.5*n*Log(d))/Log(2)-m*n;
   printf"Our packing is in dimension %o with \
    Log_2(center density) at least\n %o\n",n*m,\
    c+Log(q)/Log(2)*fromtable;

   printf"\nAsymptotic result:\n";
   l:=1000; n:=Ceiling(q^(2*l/m)*d^(1/m));
           /*Set l sufficiently large to approximate the limit*/
   L:=MinkowskiLattice(P^l);b:=Minimum(L);
   Sum:=0;
   for t in [l-1..0 by -1] do
     Z:=MinkowskiLattice(P^t);
     g:=Ceiling(Minimum(L)/Minimum(Z));
     varrho:=g/n;
      if varrho le (q-1)/q then
        Sum:=Sum+(varrho*Log(q-1)-varrho*Log(varrho)\
        -(1-varrho)*Log(1-varrho))/Log(q);
      else Sum:=Sum+1;
      end if;
   end for;

   Lambda:=-1-1/(2*m)*Log(d)/Log(2)-0.5*Log(m/(2*Pi(R)*Exp(1))) \
    /Log(2)+Log(b^0.5)/Log(2)-0.5*Log(n)/Log(2) \
    -1/m*Log(q)/Log(2)*Sum;
   printf"The asymptotic density exponent of the packing family is \
   \n at least %o",Lambda;

  else
  printf"The polynomial %o is not irreducible over Q.\n",f;
  end if;

\end{verbatim}
}


\end{document}